\documentclass[journal]{IEEEtran}
\ifCLASSINFOpdf
\else
\fi
\newcommand{\Z}{\mathbb{Z}}
\newcommand{\R}{\mathbb{R}}
\newcommand{\N}{\mathbb{N}}

\newcommand{\LM}{\mathcal{L}}

\newcommand{\dt}{\delta}
\newcommand{\Dt}{\Delta}

\newcommand{\D}{\mathcal{D}}

\newcommand{\e}{\epsilon}

\newcommand{\E}{\mathcal{E}}

\newcommand{\V}{\mathcal{V}}
\newcommand{\Gr}{\mathcal{G}}
\newcommand{\W}{\mathcal{W}}
\newcommand{\s}{\sigma}
\newcommand{\Sig}{\Sigma}

\newtheorem{theorem}{Theorem}[section]

\newtheorem{definition}{Definition}
\newtheorem{proposition}[theorem]{Proposition}

\newtheorem{example}[theorem]{Example}

\newtheorem{remark}{Remark}[section]
\newtheorem{algo}[theorem]{Algorithm}
\usepackage{times,amsmath,amssymb,graphicx,tikz,cite,enumerate}
\usetikzlibrary{automata,arrows,positioning}
\usepackage[colorlinks,linkcolor=blue,hyperindex]{hyperref}
\allowdisplaybreaks
\begin{document}
%
\title{Observability of Boolean control networks: A unified approach based
on finite automata}
%
%

\author{Kuize~Zhang, {\it Member, IEEE}, Lijun~Zhang 
\thanks{K. Zhang is with College of Automation, Harbin Engineering University, Harbin, 150001, PR China
(e-mail: zkz0017@163.com), and Institute of Systems Science, Chinese Academy of Sciences, Beijing 100190, PR China.
L. Zhang is with
School of Marine Science and Technology, Northwestern Polytechnical University, Xi'an, 710072, PR China (e-mail: zhanglj7385@nwpu.edu.cn)
and College of Automation, Harbin Engineering University, Harbin, 150001, PR China.

The original version of this paper was presented at the 33rd Chinese Control
Conference, July 28--30, 2014, Nanjing, China.
}
\thanks{This work was supported by
the Fundamental Research Funds for the Central Universities (HEUCFX41501),
National Natural Science Foundation of China (No. 61573288), Program for New Century Excellent Talents in University of Ministry of Education of China and Basic Research
Foundation of Northwestern Polytechnical University (No.
JC201230).
}
\thanks{Manuscript received xxxx xx, xxxx; revised xxxx xx, xxxx.}}

%
%

\markboth{Journal of \LaTeX\ Class Files,~Vol.~x, No.~x, xxxx~xxxx}%
{Shell \MakeLowercase{\textit{et al.}}: Bare Demo of IEEEtran.cls for Journals}
%



\maketitle

\begin{abstract}
The problem on how to determine the  observability
of Boolean control networks (BCNs)  has been open for five years already.
In this paper, we propose a unified approach
to determine all the four types of observability
of BCNs in the literature. We define the concept of weighted pair graphs for BCNs.
In the sense of each observability,
we use the so-called weighted pair graph to transform a BCN to a finite automaton, and then
we use the automaton to determine observability.
In particular, the two types of observability
that rely on initial states and inputs in the literature
are determined.
Finally, we show that no pairs of the four types of observability are equivalent,
which reveals the essence of nonlinearity of BCNs.
\end{abstract}


%
\IEEEpeerreviewmaketitle

\section{Introduction}

In 2007, Akutsu et al. \cite{Akutsu(2007)} propose the concept of {\it controllability} of
{\it Boolean control networks} (BCNs), prove that determining the controllability of BCNs is
{\bf NP}-hard\footnote{That is,
there exists no polynomial time algorithm for determining the controllability of BCNs
unless  {\bf P=NP}.},
and point out  that ``One of the major goals of
systems biology is to develop a control theory for complex biological systems''.
Since then, the study on control-theoretic problems in the areas of  Boolean
 networks (initiated by Kauffman \cite{Kauffman1969RandomBN} in 1969
to  describe
genetic regulatory networks) and Boolean control
networks (initiated in \cite{Ideker(2001)} in 2001) has drawn vast attention (cf.
\cite{Cheng2009bn_ControlObserva,Cheng2011IdentificationBCN,Zhao2010InputStateIncidenceMatrix,Cheng(2011book),Fornasini2013ObservabilityReconstructibilityofBCN,Laschov2013ObservabilityofBN:GraphApproach,Zhang2013ControlObservaBCNTVD,Li2013ObservabilityConditionsofBCN,Zhang2015Invertibility:BCN} etc.).
Controllability and {\it observability} are basic control-theoretic problems.
In 2009, Cheng et al. \cite{Cheng2009bn_ControlObserva}
construct a control-theoretic
framework for BCNs by using a new tool, called the {\it semi-tensor product} (STP) of matrices
proposed in \cite{Cheng2001STP} in 2001,
and give equivalent conditions for controllability
of BCNs and observability of controllable BCNs. Since then,  to the best of our knowledge,
how to determine this observability  has been open.
This type of observability means that
every initial state can be determined by an input sequence. 
Later on, important results on other types of observability
of BCNs came up.
Until now, there are four types of
observability. 
Another observability, proposed  in
\cite{Zhao2010InputStateIncidenceMatrix} in 2010, stands for that
for every two distinct initial states,
there exists an input sequence which can distinguish them.
There is a sufficient but not necessary condition in \cite{Zhao2010InputStateIncidenceMatrix}.
However, there is no equivalent condition in \cite{Zhao2010InputStateIncidenceMatrix}.
This observability is determined in \cite{Li2015ControlObservaBCN} in 2015
based on an algebraic method.
A third observability stating that there is an input sequence
that determines the initial state, is proposed in \cite{Cheng2011IdentificationBCN}
to study the identifiability problem of BCNs in 2011.
It is proved that determining this observability is {\bf NP}-hard in
\cite{Laschov2013ObservabilityofBN:GraphApproach} in 2013.
Nevertheless, one way is proposed in  \cite{Li2013ObservabilityConditionsofBCN}
to determine this observability in 2013.
A fourth observability is determined in
\cite{Fornasini2013ObservabilityReconstructibilityofBCN,Xu2013ObserverFA_STP}\footnote{Note
that  the types of observability studied in \cite{Fornasini2013ObservabilityReconstructibilityofBCN,Xu2013ObserverFA_STP}
are the same.} in 2013, which is
essentially the observability of linear control systems,
i.e., every sufficiently long input sequence can
determine the initial state.

Like nonlinear systems,
BCNs are polynomial systems over ${\mathbb F}_2$, the Galois field of two elements \cite{Li2015ControlObservaBCN}.
This explains why
the observability proposed in \cite{Cheng2009bn_ControlObserva,Zhao2010InputStateIncidenceMatrix}
that rely on initial states and inputs are important for BCNs.
The methods of determining the last two types of observability
are not suitable for the first two, mainly because they
are based on the independence of
initial states and/or inputs. Besides, it is not known whether the method for
the second type used in
\cite{Li2015ControlObservaBCN} is suitable for the other three types now.
In this paper, we  propose
a unified method based on {\it finite automata}
to determine all the four types of observability regardless of dependence.
To this end, we firstly define
{\it weighted pair graphs} for BCNs, which consist of pairs of states of BCNs producing
the same outputs, and transitions between the pairs.
Secondly, we use the weighed pair graph to transform a BCN to a deterministic finite
automaton. Finally, we use the automaton to determine observability.

The remainder of this paper is organized as follows. Section \ref{sec2}
introduces necessary preliminaries about
STP, BCNs with their algebraic forms, {\it formal languages} and finite automata.
Section \ref{sec3} presents the algorithms to determine all the four types of observability.
Section \ref{sec4} shows the pairwise nonequivalence of  
the four types of observability of BCNs.
Section \ref{sec6} ends up with some remarks and challenging open problems.

\section{Preliminaries}\label{sec2}

\subsection{The semi-tensor product of matrices}

We first introduce some related notations in STP.

\begin{itemize}
	\item $2^A$: the power set of  set $A$
  \item $\Z_+$: the set of  positive integers
  \item $\N$: the set of  natural numbers
  \item $\D$: the set
  $\{0,1\}$
  \item $\delta_n^i$: the $i$-th column of the identity matrix $I_n$
  \item $\Delta_n$: the set $\{\delta_n^1,\dots,\delta_n^n\}$
	  ($\Dt:=\Dt_2$)
  \item $\delta_n[i_1,\dots,i_s]$: the {\it logical matrix}
	  $[\delta_n^{i_1},\dots,\delta_n^{i_s}]$ ($i_1,\dots,i_s\in\{1,2,\dots, n\}$)
	  (for the concept of logical matrices, we refer the reader to  \cite{Cheng(2011book)}.)
  \item $\LM_{n\times s}$: the set of
  $n\times s$ logical matrices, i.e., $\{\delta_n[i_1,\dots,i_s]|i_1,\dots,i_s\in\{1,2,\dots,n\}\}$
  \item $[1,N]$: the first $N$ positive integers
  \item $|A|$: the cardinality of set $A$
\end{itemize}

\begin{definition}\cite{Cheng(2011book)}
	Let $A\in \R_{m\times n}$, $B\in \R
	_{p\times q}$, and $\alpha=\mbox{lcm}
	(n,p)$ be the least common multiple of $n$ and $p$. The STP of $A$
	and $B$ is defined as
		$A\ltimes B = (A\otimes I_{\frac{\alpha}{n}})(B\otimes I_{\frac
		{\alpha}{p}})$,
	where $\otimes$ denotes the Kronecker product.
\end{definition}

From this definition, it is easy to see that the conventional
product of matrices is a particular case of STP.
Since STP keeps most properties of the conventional product \cite{Cheng(2011book)},
e.g., the associative law, the distributive law,
etc.,
we usually omit the symbol ``$\ltimes$'' hereinafter.

\subsection{Boolean control networks and their algebraic forms}

In this paper, we investigate the following BCN
with $n$ state nodes, $m$ input nodes and $q$ output nodes:
\begin{equation}\label{BCN1}
\begin{split}
 &x (t + 1) = f (u (t),x (t)), \\
 &y(t)=h(x (t)),\\
 \end{split}
\end{equation}
where $x\in\D^n$; $u\in\D^m$; $y\in\D^q$;
$t=0,1,\dots$; $f:\D^{n+m}\to\D^n$ and $h:\D^n\to\D^q$ are logical
functions.

Using the STP of
matrices, (\ref{BCN1}) can be  equivalently represented in
the following algebraic form \cite{Cheng2009bn_ControlObserva}
\begin{equation}\label{BCN2}
\begin{split}
 &x(t + 1) = Lu(t)x(t),\\
 &y(t) = Hx(t),
\end{split}
\end{equation}
where $x\in\Delta_{N}$, $u\in\Delta_{M}$ and $y\in\Dt_{Q}$ denote
states, inputs and outputs, respectively;
$t=0,1,\dots$; $L\in\LM_{N\times (NM)}$; $H\in\LM_{Q\times N}$;
hereinafter, $N:=2^n$, $M:=2^m$ and $Q:=2^q$.

For more details on properties of STP,
and how to transform a BCN into its equivalent algebraic form,
we  refer the reader to \cite{Cheng2009bn_ControlObserva}.

\subsection{Formal languages and finite automata}

The theories of formal languages and finite automata are
among the mathematical foundations of theoretical computer
science \cite{Kari2013LectureNoteonAFL}.
Let $\Sig$ be a finite nonempty set (called {\it alphabet}).
We use $\Sig^*$ to denote the set of all finite sequences (called {\it words}) of
elements (called {\it letters}) of $\Sig$.
The empty word is denoted by $\epsilon$.
$|u|$ denotes the length of word $u$. For example, $|abc|=3$ for
the alphabet $\{a,b,c\}$, $|\epsilon|=0$. The set of all words of length $p$
is denoted by $\Sig^p$. Notice that $\Sig^{0}=\{\epsilon\}$.
Then $\Sig^{*}=\cup_{p=0}^{\infty}\Sig^{p}$.
A formal language (or language for short) is a subset of $\Sig^*$.

A deterministic finite automaton (DFA) is defined as
5-tuple $A=(S,\Sig,\s,s_0,F)$, where $S$ denotes the finite state set,
$\Sig$  the finite alphabet,
$s_0\in S$  the initial state, $F\subset S$  the final state set,
and  $\s:S\times\Sig\to S$  the transition partial function,
i.e., a function defined on a fixed subset of $S\times\Sig$,
which can naturally be
extended to $\s:S\times \Sig^*\to S$.
We call a
DFA {\it complete} if $\s$ is a function from $S\times\Sig^*$ to $S$.
A language $L$ over alphabet $\Sig$ is called {\it regular}, if it is {\it recognized} by
a DFA $A=(S,\Sig,\s,s_0,F)$, i.e., $L=\{w\in\Sig^*|\s(s_0,w)\in F\}$.
A word $u\in\Sig^*$ such that $\s(s_0,u)\in F$ is called {\it accepted} by DFA $A$.
A DFA accepts the empty word $\epsilon$ iff its initial state is final.

In order to represent a DFA,
we introduce the transition graph of DFA
$A=(S,\Sig,\s,s_0,F)$.
Let $V,E$ and $W$ be the vertex set, the edge set and the weight function of
a weighted directed graph $G=(V,E,W)$.
$G$ is called the transition graph of DFA $A$,
if $V=S$, $E=\{(s_i,s_j)\in V\times V|\text{there is }a\in\Sig\text{ such that }
\s(s_i,a)=s_j\}\subset V\times V$, and
$W:E\to 2^\Sig$, $(s_i,s_j)\mapsto \{a\in\Sig|\s(s_i,a)=s_j\}$. 

In the transition graph of a DFA, we add a ``start'' input arrow to the vertex of
the initial state, and use double circles  to denote
final states. We omit the curly bracket ``$\{\}$''
in the weights of edges. See Fig. \ref{fig4:observability} for an example.

Now we give a proposition on finite automata that will be used in
the main results.

\begin{proposition}\label{prop6_observability}
	Given a DFA $A=(S,\Sig,\s,s_0,F)$. Assume that $F=S$ and for each
	$s\in S$, there is a word $u\in\Sig^*$ such that $\s(s_0,u)=s$.
	Then $L(A)=\Sig^*$ iff $A$ is complete.
\end{proposition}

\begin{proof}
	``if'': If $A$ is complete and $F=S$, then $\e\in L(A)$ and
	for any nonempty word $w\in\Sig^*$, $\s(s_0,w)\in F$, i.e., $w\in
	L(A)$. Hence $L(A)=\Sig^*$.

	``only if'': Assume that $F=S$ and $A$ is not complete.
	Choose an $s\in S$ such that $\s$ is not well defined at $(s,a)$
	for some $a\in\Sig$.
	Choose  word $w\in \Sig^*$ such that $\s(s_0,w)=s$, 
	then
	word $wa\notin L(A)$, for $A$ is deterministic.
	That is, $L(A)\ne \Sig^*$.
\end{proof}

\section{Determining the observability of BCNs}\label{sec3}

\subsection{Weighted pair graph}

In this subsection, we define a weighted directed graph for BCN \eqref{BCN2},
named weighted pair graph.
Based on the weighted pair graph, in the following subsections,
we  construct a DFA in the sense of each observability, and then
use the  obtained DFA and Proposition \ref{prop6_observability}
to determine  observability.

\begin{definition}\label{pairgraph}
	Consider BCN \eqref{BCN2}.
	Let $\V,\E$ and $\W$ be the vertex set, the edge set and the weight function of
	a weighted directed graph $\Gr=(\V,\E,\W)$. $\Gr$ is
	called the weighted pair graph of the BCN, if
	$\V = \{(x,x')\in\Dt_N\times\Dt_N|Hx=Hx'\}$\footnote{Here
	$(x,x')$ is an unordered pair, i.e., $(x,x')=(x',x)$.},
	$\E=\{((x_1,x_1'),(x_2,x_2'))\in\V\times\V|
	\text{there exists }u\in\Dt_M\text{ such that }Lux_1=x_2\text{ and }Lux_1'=x_2',
	\text{ or, }Lux_1=x_2'\text{ and }Lux_1'=x_2\}\subset \V\times\V$, and
	$\W:\E\to 2^{\Dt_M}$, $((x_1,x_1'),(x_2,x_2'))\mapsto
	\{u\in\Dt_M|Lu_1x_1=x_2\text{ and }Lu_1x_1'=x_2'\text{, or, }
	Lux_1=x_2'\text{ and }Lux_1'=x_2\}$. 
\end{definition}

Intuitively, there is an edge from a vertex $v$ to another one $v'$, iff there is
an input $u$ driving one state in $v$ to one state in $v'$ and driving the other
state in $v$ to the other state in $v'$.
Similar to the transition graph of a DFA, we omit the curly
bracket ``$\{\}$'' in the weights of edges.
Hereinafter, we call each vertex $(x,x)\in\Dt_N\times\Dt_N$ a diagonal
vertex.

From Definition \ref{pairgraph}, the weighted pair graph
consists of every state pair producing the same output.
In  fact, to test whether a BCN is observable,
is to test whether these states can be distinguished by
input sequences.

Let $(\V,\E,\W)$ be a weighted pair graph. For  a subset $V$ of $\V$, the subgraph
generated by $V$ is the graph $(V,\E\cap (V\times V),\W|_{\E\cap (V\times V)})$,
where $\W|_{\E\cap (V\times V)}$ is the restriction of $\W$ to $\E\cap (V\times V)$.

The weighted pair graph of the following BCN \eqref{eqn2_observability} is  depicted in
Fig. \ref{fig3:observability}.

	\begin{equation}
		\begin{split}
			x(t+1) &= \dt_4[1,1,2,1,2,4,1,1]x(t)u(t),\\
			y(t) &= \dt_2[1,2,2,2]x(t),
		\end{split}
		\label{eqn2_observability}
	\end{equation}
	where $t\in\N$, $x\in\Dt_4$, $y,u\in\Dt$.

\begin{figure}
        \centering
\begin{tikzpicture}[>=stealth',shorten >=1pt,auto,node distance=1.5 cm, scale = 0.8, transform shape,
	->,>=stealth,inner sep=2pt,state/.style={shape=circle,draw,top color=red!10,bottom color=blue!30},
	point/.style={circle,inner sep=0pt,minimum size=2pt,fill=},
	skip loop/.style={to path={-- ++(0,#1) -| (\tikztotarget)}}]
	\node[state] (11)                                 {$11$};
	\node[state] (22) [right of = 11]                    {$22$};
	\node[state] (42) [left of = 11]                 {$24$};
	\node[state] (32) [right of = 22]                 {$23$};
	\node[state] (44) [below of = 11]                 {$44$};
	\node[state] (33) [below of = 22]                 {$33$};
	\node[state] (34) [left of = 44]                 {$34$};

	\path [->] (11) edge [loop above] node {$1,2$} (11)
		   (22) edge [loop above] node {$1$} (22)
		   (22) edge node {$2$} (11)
		   (42) edge node {$2$} (11)
		   (32) edge node {$1$} (22)
		   (44) edge node {$1,2$} (11)
		   (33) edge node {$1$} (22)
		   (33) edge node {$2$} (44)
		   ;
        \end{tikzpicture}
	\caption{The weighted pair graph of BCN \eqref{eqn2_observability},
	where the number $ij$ in each circle denotes the state pair $(\dt_4^i,
	\dt_4^j)$, and the weight $k_1,k_2,\dots$ beside each edge denotes the
	weight $\{\dt_2^{k_1},\dt_2^{k_2},\dots\}$ of the edge.}
	\label{fig3:observability}
\end{figure}
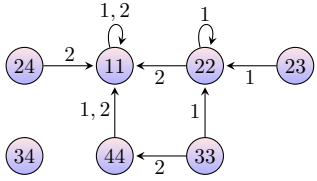

\subsection{Notations}

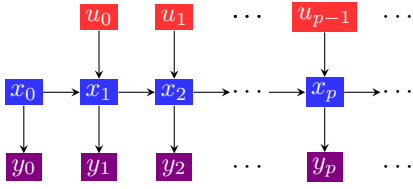
\begin{figure}
        \centering
        \begin{tikzpicture}[->,>=stealth,node distance=1.0cm]
          \tikzstyle{state}=[shape=rectangle,fill=blue!80,draw=none,text=white,inner sep=2pt]
          \tikzstyle{input}=[shape=rectangle,fill=red!80,draw=none,text=white,inner sep=2pt]
          \tikzstyle{output}=[shape=rectangle,fill=red!50!blue,draw=none,text=white,inner sep=2pt]
          \tikzstyle{disturbance}=[shape=rectangle,fill=red!20!blue,draw=none,text=white,inner sep=2pt]
	  \tikzstyle{time}=[inner sep=0pt,minimum size=5mm]
	  \node [state] (x0)              {$x_0$};
          \node [state] (x1) [right of=x0] {$x_1$};
	  \node [state] (x2) [right of=x1] {$x_2$};
	  \node [input] (u0) [above of=x1] {$u_0$};
	  \node [input] (u1) [right of=u0] {$u_1$};
	  \node [time] (x3) [right of=x2] {$\cdots$};
	  \node [time] (u2) [above of=x3] {$\cdots$};
	  \node [state] (x4) [right of=x3] {$x_p$};
	  \node [time] (u2) [above of=x3] {$\cdots$};
	  \node [input] (u3) [above of=x4] {$u_{p-1}$};
	  \node [time] (x5) [right of=x4] {$\cdots$};
	  \node [time] (u4) [above of=x5] {$\cdots$};
	  \node [output] (y0) [below of=x0] {$y_0$};
	  \node [output] (y1) [below of=x1] {$y_1$};
	  \node [output] (y2) [below of=x2] {$y_2$};
	  \node [output] (y4) [below of=x4] {$y_p$};
	  \node [time] (y3) [below of=x3] {$\cdots$};
	  \node [time] (y5) [below of=x5] {$\cdots$};
          \path (x0) edge (x1)
	        (x1) edge (x2)
		(u0) edge (x1)
		(u1) edge (x2)
		(x2) edge (x3)
		(x3) edge (x4)
		(u3) edge (x4)
		(x4) edge (x5)
		(x0) edge (y0)
		(x1) edge (y1)
		(x2) edge (y2)
		(x4) edge (y4);
        \end{tikzpicture}
	\caption{The input-state-output-time transfer graph of BCN
	\eqref{BCN2}, where
		subscripts stand for time steps,
		$x_0,x_1,\dots$  states, $u_0,u_1,\dots$  inputs,
		$y_1,y_2,\dots$  outputs, and arrows infer dependence.}
    \label{fig:input-state-output-time-graph}
\end{figure}

The input-state-output-time transfer graph
of BCN \eqref{BCN2} is drawn in Fig.
\ref{fig:input-state-output-time-graph}.
In order to define  these observability, we
define the following mappings: 

Let $\Dt_M,\Dt_N,\Dt_Q$ be three alphabets.
For all $x_0\in\Dt_N$ and all $p\in\Z_{+}$,
\begin{enumerate}
	\item
		\begin{equation}\label{lx0n}
			\begin{split}
				&L_{x_0}^p:(\Dt_M)^{p}\to(\Dt_N)^{p},u_0
				\dots u_{p-1}\mapsto x_1\dots x_p,\\
			&L_{x_0}^{\N}:(\Dt_M)^{\N}\to(\Dt_N)^{\N},u_0u_1\dots
			\mapsto	x_1x_2\dots.
			\end{split}
		\end{equation}
	\item
		\begin{equation}\label{hx0n}
			\begin{split}
				&(HL)_{x_0}^p:(\Dt_M)^{p}\to(\Dt_Q)^{p},u_0
				\dots u_{p-1}\mapsto y_1\dots y_p,\\
				&(HL)_{x_0}^\N: (\Dt_{M})^{\N}\to(\Dt_{Q})^{\N},
			u_0u_1\dots\mapsto y_1y_2\dots.
			\end{split}
		\end{equation}
\end{enumerate}

For all  $p\in\Z_{+}$,
all $U=u_1\dots u_p\in(\Dt_M)^{p}$,
and all $1\le i\le j\le|U|$, we use $U[i,j]$ to denote the word $u_i\dots u_j$.
In particular, $U[i]$ (or $U(i)$) is short for $U[i,i]$.
Given $U\in(\Dt_M)^*$, $U^{\infty}$ denotes the {\it concatenation} of infinite copies
of $U$, i.e., $UU\dots$. For all input sequences $U=u_0u_1\dots$$\in(\Dt_M)^{\N}$, and
all $0\le i\le j\in\N$, we use $U[i,j]$ to denote the word $u_i\dots u_j$.

\subsection{Determining the observability  in \cite{Cheng2009bn_ControlObserva}}

\begin{definition}[\cite{Cheng2009bn_ControlObserva}]\label{def1_observability}
	BCN \eqref{BCN2} is called observable, if for every initial state $x_0
	\in\Dt_N$, there exists an input sequence such that the initial state
	can be determined by the output sequence.
\end{definition}

Definition \ref{def1_observability} can be expressed equivalently as follows:

\begin{definition}\label{def3_observability}
	BCN \eqref{BCN2} is called observable, if for every initial state $x_0
	\in\Dt_N$, there exists an input sequence $U\in(\Dt_M)^{p}$ for some
	$p\in\Z_{+}$ such
	that for all states $x_0\ne\bar x_0\in\Dt_N$,
	$Hx_0=H\bar x_0$ implies $(HL)_{x_0}^{p}(U)\ne (HL)_{\bar
	x_0}^{p}(U)$.
\end{definition}

In this subsection, the observability of BCN \eqref{BCN2} refers to
 Definition \ref{def3_observability}.

According to Definition \ref{def3_observability},
BCN \eqref{BCN2} is not observable iff there is a state $\dt_N^i$
in a non-diagonal vertex of its weighted pair graph $\Gr=(\V,\E,\W)$
such that
for all $p\in\Z_{+}$, all $U\in(\Dt_M)^p$, there is a state $\dt_N^j$
with  $j\ne i$, $(\dt_N^i,\dt_N^j)\in \V$ and $(HL)_{\dt_N^i}^p(U)=(HL)_{\dt_N^j}^p(U)$.

For fixed $\dt_N^i$, we design an algorithm to construct a DFA
for BCN \eqref{BCN2} according to its weighted pair graph $(\V,\E,\W)$.
The new DFA is denoted by $A_{\dt_N^i}$, and accepts exactly all finite input
sequences that do not determine $\dt_N^i$.
The states of DFA $A_{\dt_N^i}$ are subsets of $\V$.

\begin{algo}\label{alg1_observability}
\begin{enumerate}
	\item\label{item1_obser1}
		Set $\Dt_M$ to be the alphabet of the DFA.
		Set the subset of $\V$ consisting of all the
		non-diagonal vertices of $\V$
		that contain $\dt_N^i$ to be the initial state of the DFA.
		That is, the set $s_0:=\{(\dt_N^k,\dt_N^l)|k,l\in[1,N],
		H\dt_N^k=H\dt_N^l,
		k\ne l, k\text{ or }l=i\}$ is   the initial state
		of the DFA.
	\item\label{item2_obser1}
		For each letter $\dt_M^j$, $j\in[1,M]$, find the value
		for the transition partial
		function of the DFA at $(s_0,\dt_M^j)$. 
		The specific procedure is as follows:
		
		Fix $j\in[1,M]$. Set $s_j:=\{v\in\V|\text{there is }v'\in s_0
		\text{ such that }(v',v)\in\E,\text{ and }\dt_M^j\in\W((v',v))\}$.
		If $s_j\ne\emptyset$, add $s_j$ to the state
		set of the DFA and set $s_j$ to be the value of the transition
		partial function at $(s_0,\dt_M^j)$;
		otherise,  the transition partial function
		is not well defined at
		$(s_0,\dt_M^j)$.

	\item\label{item3_obser1}
		For each new  state $s$ of the DFA found in the previous step,
		and each letter $\dt_M^j$, $j\in[1,M]$, find the value
		for the transition partial
		function at $(s,\dt_M^j)$ according to
		Step \ref{item2_obser1}.
	\item\label{item4_obser1}
		Repeat Step \ref{item3_obser1} until
	    no new state of the DFA occurs. (Since $\V$
		is a finite set, so is its power set,
		this repetition  will stop.)
	\item\label{item5_obser1}
		Set all the states of the DFA to be final states.
\end{enumerate}
\end{algo}

Take BCN \eqref{eqn2_observability} for example. Choose state $\dt_4^2$.
Then the DFA $A_{\dt_4^2}$ generated by Algorithm \ref{alg1_observability}
is as shown in Fig. \ref{fig4:observability}.

\begin{figure}
		\centering
\begin{tikzpicture}[>=stealth',shorten >=1pt,auto,node distance=2.0 cm, scale = 1.0, transform shape,
	->,>=stealth,inner sep=2pt,state/.style={shape=circle,draw,top color=red!10,bottom color=blue!30},
	point/.style={circle,inner sep=0pt,minimum size=2pt,fill=},
	skip loop/.style={to path={-- ++(0,#1) -| (\tikztotarget)}}]
	\node[accepting,initial,state] (2*)  {$23,24$};
	\node[accepting,state] [above right of = 2*] (11) {$11$};
	\node[accepting,state] [below right of = 2*] (22) {$22$};
	\path [->] (2*) edge node {$2$} (11)
		   (2*) edge node {$1$} (22)
		   (22) edge [bend right] node {$2$} (11)
		   (22) edge [loop right] node {$1$} (22)
		   (11) edge [loop right] node {$1,2$} (11)
	;
		\end{tikzpicture}
	\caption{The DFA $A_{\dt_4^2}$ with respect to BCN
	\eqref{eqn2_observability} generated by Algorithm
	\ref{alg1_observability},
	where the number $ij$ in each circle denotes the state pair $(\dt_4^i,
	\dt_4^j)$, and the weight $k$ beside each edge denotes the input $\dt_2^k$.
	}
	\label{fig4:observability}
\end{figure}

Now we give a necessary and sufficient condition for this observability.

\begin{theorem}\label{thm3_observability}
  BCN \eqref{BCN2} is not observable in the sense of
  Definition \ref{def3_observability} iff there is a state
  $\dt_N^i$ in a non-diagonal
  vertex of its weighted pair graph such that
  the DFA $A_{\dt_N^i}$ generated by Algorithm \ref{alg1_observability}
  recognizes language $(\Dt_M)^*$.

\end{theorem}

\begin{proof}
  ``only if'':
  Assume that BCN \eqref{BCN2} is not observable, then there is
  a state $\dt_N^i$ such that for all $p\in\Z_{+}$, all
  $U\in(\Dt_M)^p$, there is a state $\dt_N^j$ satisfying $i\ne j$,
  $H\dt_N^i=H\dt_N^j$, and $(HL)_{\dt_N^i}^p(U)=(HL)_{\dt_N^j}^p(U)$.
  According to Algorithm \ref{alg1_observability}, $v_0:=(\dt_N^i,\dt_N^j)$
  is in the initial state of DFA $A_{\dt_N^i}$. Denote the weighted pair graph
  of BCN \eqref{BCN2} by $\Gr=(\V,\E,\W)$.
  Then there exist vertices $v_k:=(\dt_N^{i_k},\dt_N^{j_k})\in\V$
  such that
  $U[k]\in \W( (v_{k-1},v_{k}))$, $k=1,\dots,p$.
  That is, for all $p\in\Z_{+}$, each $U$ in $(\Dt_M)^p$ is accepted by
  DFA  $A_{\dt_N^i}$.
  It is obvious that $\e\in L (A_{\dt_N^i})$.
  Then $L (A_{\dt_N^i})=(\Dt_M)^*$.

  ``if'':
  Note that the DFA $A_{\dt_N^i}$ accepts exactly all finite input
  sequences that do not determine $\dt_N^i$. Then $L (A_{\dt_N^i})=(\Dt_M)^*$
  implies that for all $p\in\Z_{+}$, all
  $U\in(\Dt_M)^p$, there is a state $\dt_N^j$ such that $i\ne j$,
  $H\dt_N^i=H\dt_N^j$, and $(HL)_{\dt_N^i}^p(U)=(HL)_{\dt_N^j}^p(U)$.
  That is, BCN \eqref{BCN2} is not observable.
\end{proof}

Proposition \ref{prop6_observability},
Theorem \ref{thm3_observability} and Algorithm
\ref{alg1_observability} directly imply  the following result that
can be used to check whether a given BCN is observable.

\begin{theorem}\label{alg2_observability}
  BCN \eqref{BCN2} is not observable in the sense of
  Definition \ref{def3_observability} iff there is a state
  $\dt_N^i$ in a non-diagonal
  vertex of its weighted pair graph such that
  the DFA $A_{\dt_N^i}$ generated by Algorithm \ref{alg1_observability}
  is complete.
\end{theorem}

\begin{example}\label{exam2_observability}
	Check whether BCN \eqref{eqn2_observability} is observable.
	
	According to Theorem \ref{alg2_observability}, we should check
	$\dt_4^2,\dt_4^3,\dt_4^4$ one by one.

	First we check $\dt_4^2$. According to Algorithm
	\ref{alg1_observability}, we calculate DFA $A_{\dt_4^2}$, and
	derive the transition graph of this DFA as shown in Fig.
	\ref{fig4:observability}. This DFA is complete, by
	Theorem \ref{alg2_observability}, BCN \eqref{eqn2_observability} is
	not observable.
\end{example}

\subsection{Determining the observability  in \cite{Zhao2010InputStateIncidenceMatrix}}

\begin{definition}\label{def4_observability}
	BCN \eqref{BCN2} is called observable, if for any distinct states
	$x_0,\bar x_0\in\Dt_N$, there is an input sequence $U\in(\Dt_M)^p$
	for some $p\in\Z_{+}$, such that $Hx_0=H\bar x_0$ implies
	$(HL)_{x_0}^{p}(U)\ne (HL)_{\bar x_0}^{p}(U)$.\footnote{Actually,
	after removing ``$Hx_0=H\bar x_0$ implies'' in Definition \ref{def4_observability},
	Definition \ref{def4_observability} becomes
	the observability studied in \cite{Zhao2010InputStateIncidenceMatrix}.
	In order to make the observability studied in \cite{Zhao2010InputStateIncidenceMatrix}
	exactly the widely accepted one for nonlinear control systems, we modify it in
	Definition \ref{def4_observability}.}
\end{definition}

In this subsection, the observability of BCN \eqref{BCN2} means Definition \ref{def4_observability}.

According to Definition \ref{def4_observability},
BCN \eqref{BCN2} is not observable iff there is a non-diagonal vertex
$(\dt_N^i,\dt_N^j)$ in its weighted pair graph  such that
for all $p\in\Z_{+}$, and $U\in(\Dt_M)^p$,
$(HL)_{\dt_N^i}^p(U)=(HL)_{\dt_N^j}^p(U)$.

For a fixed non-diagonal vertex $(\dt_N^i,\dt_N^j)$,
we design an algorithm to construct a DFA
for BCN \eqref{BCN2} according to its weighted pair graph.
The new DFA is denoted by  $A_{(\dt_N^i,\dt_N^j)}$,
and accepts exactly all finite input sequences that do not
distinguish $\dt_N^i$ and $\dt_N^j$.

\begin{algo}\label{alg3_observability}
	\begin{enumerate}
		\item Set $\Dt_M$ to be the alphabet of the DFA.
		Set vertex $(\dt_N^i,\dt_N^j)$
		to be the initial state of the DFA.
		\item Find each vertex $v$ such that there is a path
			from $(\dt_N^i,\dt_N^j)$ to $v$.
			Keep the subgraph generated by
			$(\dt_N^i,\dt_N^j)$ and those vertices, and remove all vertices and
			edges outside of the subgraph.
		\item Set each remainder vertex to be a final state of the DFA.
	\end{enumerate}
\end{algo}

Again take BCN \eqref{eqn2_observability} as an example.
 The DFA of each non-diagonal vertex of the weighted pair
graph generated by Algorithm \ref{alg3_observability}
is shown in Fig. \ref{fig5:observability}.

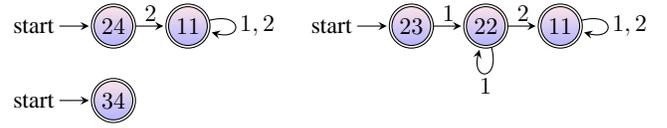
\begin{figure}
		\centering
\begin{tikzpicture}[>=stealth',shorten >=1pt,auto,node distance=1.1 cm, scale = 0.9, transform shape,
	->,>=stealth,inner sep=2pt,state/.style={shape=circle,draw,top color=red!10,bottom color=blue!30},
	point/.style={circle,inner sep=0pt,minimum size=2pt,fill=},
	skip loop/.style={to path={-- ++(0,#1) -| (\tikztotarget)}}]
	\tikzstyle{time}=[inner sep=0pt,minimum size=5mm]

	\node[accepting,initial,state] (241)  {$24$};
	\node[accepting, state] (111) [right of = 241] {$11$};
	\path [->] (241) edge node {$2$} (111)
	(111) edge [loop right] node {$1,2$} (111)
	;

	\node[time] (1) [right of = 111] {};
	\node[time] (2) [right of = 1] {};

	\node[accepting,initial,state] (232) [right of = 2] {$23$};
	\node[accepting, state] (222) [right of = 232] {$22$};
	\node[accepting, state] (112) [right of = 222] {$11$};
	\path [->] (232) edge node {$1$} (222)
	(222) edge [loop below] node {$1$} (222)
	(222) edge node {$2$} (112)
	(112) edge [loop right] node {$1,2$} (112)
	;

	\node[accepting,initial,state] (345) [below of = 241] {$34$};

		\end{tikzpicture}
	\caption{The DFA of each non-diagonal vertex of the weighted pair
	graph of BCN \eqref{eqn2_observability}
	generated by Algorithm \ref{alg3_observability},
	where the number $ij$ in each circle denotes the state pair $(\dt_4^i,
	\dt_4^j)$, and the weight $k$ beside each edge denotes the input $\dt_2^k$.
	}
	\label{fig5:observability}
\end{figure}

The following is a necessary and sufficient condition for this observability.

\begin{theorem}\label{thm4_observability}
  BCN \eqref{BCN2} is not observable in the sense of
  Definition \ref{def4_observability} iff there is a non-diagonal vertex
  $(\dt_N^i,\dt_N^j)$ in its weighted pair graph such that
  the DFA $A_{(\dt_N^i,\dt_N^j)}$ generated by Algorithm
  \ref{alg3_observability}
  recognizes language $(\Dt_M)^*$.

\end{theorem}

\begin{proof}
  ``only if'':
  Assume that BCN \eqref{BCN2} is not observable, then there is
  a non-diagonal vertex $(\dt_N^i,\dt_N^j)$ in the weighted pair graph of
  BCN \eqref{BCN2} such that  for all $p\in\Z_{+}$, all $U\in(\Dt_M)^p$,
  $(HL)_{\dt_N^i}^p(U)=(HL)_{\dt_N^j}^p(U)$.
  Then for all $p\in\Z_{+}$,  each  $U$ in $(\Dt_M)^p$ is accepted by DFA
  $A_{(\dt_N^i,\dt_N^j)}$.
  It is obvious that $\e\in L (A_{(\dt_N^i,\dt_N^j)})$.
  Then $L (A_{(\dt_N^i,\dt_N^j)})=(\Dt_M)^*$.

  ``if'':
  Obvious by  Definition \ref{def4_observability}.
\end{proof}

From Proposition \ref{prop6_observability},
Theorem \ref{thm4_observability} and Algorithm
\ref{alg3_observability}, the following result which follows
can be used to determine whether a given
BCN is observable.

\begin{theorem}\label{alg4_observability}
  BCN \eqref{BCN2} is not observable in the sense of
  Definition \ref{def4_observability} iff there is a non-diagonal vertex
  $(\dt_N^i,\dt_N^j)$ in its weighted pair graph such that
  the DFA $A_{(\dt_N^i,\dt_N^j)}$ generated by Algorithm
  \ref{alg3_observability} is complete.
\end{theorem}

\begin{example}\label{exam3_observability}
	Check whether BCN \eqref{eqn2_observability} is observable.	
	According to Theorem \ref{alg4_observability}, one should check
	$(\dt_4^2,\dt_4^3),(\dt_4^2,\dt_4^4),(\dt_4^3,\dt_4^4)$
	one by one.
	From Fig. \ref{fig5:observability}, one sees that
	$\dt_2^2\notin L(A_{(\dt_4^2,\dt_4^3)})$,
	$\dt_2^1\notin L(A_{(\dt_4^2,\dt_4^4)})$, and
	$\dt_2^1,\dt_2^2\notin L(A_{(\dt_4^3,\dt_4^4)})$.
	Then by Theorem \ref{alg4_observability}, BCN \eqref{eqn2_observability}
	is observable.
\end{example}

At the end of this subsection,
using the concept of weighted pair graphs, we give a further result on this
observability.

\begin{theorem}
	Consider BCN \eqref{BCN2}. Denote
	the number of non-diagonal vertices
	of its weighted pair graph by $N_{nd}$.
	The following two items are equivalent.
	\begin{enumerate}[(i)]
		\item \label{item2_observability_length}
			The BCN is observable in the sense of Definition \ref{def4_observability}.
		\item \label{item1_observability_length}
			$N_{nd}=0$ or
		for all distinct states $x_0,\bar x_0\in\Dt_N$, there is an input sequence
		$U\in(\Dt_M)^{N_{nd}}$ such that $Hx_0=H\bar x_0$ implies
		$(HL)_{x_0}^{N_{nd}}(U)\ne(HL)_{\bar x_0}^{N_{nd}}(U)$.
	\end{enumerate}
	\label{thm:length:NewObservability}
\end{theorem}

\begin{proof}
    \eqref{item1_observability_length} $\Rightarrow$
    \eqref{item2_observability_length}:

    Obvious by  Definition
	\ref{def4_observability}.

    \eqref{item2_observability_length} $\Rightarrow$ \eqref{item1_observability_length}:

	Assume that \eqref{item1_observability_length} does not hold.
	That is, $N_{nd}>0$ and there are distinct $x,x'\in\Dt_N$,
	for all $U\in(\Dt_M)^{N_{nd}}$,
	$Hx=Hx'$ and $(HL)_{x}^{N_{nd}}(U)=(HL)_{x'}^{N_{nd}}(U)$.
	Use Algorithm \ref{alg3_observability} to generate DFA $A_{(x,x')}=(S,\Dt_M,\s,
	(x,x'),S)$. Then $\cup_{i=0}^{N_{nd}}(\Dt_M)^i\subset L(A_{(x,x')})$.
	We claim that $A_{(x,x')}$ is complete. Suppose the contrary: there is a state
	$v$ of $A_{(x,x')}$ and an input $u\in\Dt_M$ such that $\s$ is not well defined
	at $(v,u)$. Then $v$ is a non-diagonal vertex of the weighed pair graph,
	because for all diagonal vertices $v'$ (if exist), for all inputs $u'\in\Dt_M$,
	$\s$ is well defined at $(v',u')$.
    There are exactly $N_{nd}$ non-diagonal vertices, then there exists an input sequence
	$U_1$ of length less than $N_{nd}$ such that $\s((x,x'),U_1)=v$.
	We get a contradiction $U_1u\in\cup_{i=0}^{N_{nd}}(\Dt_M)^i\setminus L(A_{(x,x')}) $.
	By Theorem \ref{alg4_observability}, the BCN is not observable.
\end{proof}

\subsection{Determining the observability  in \cite{Cheng2011IdentificationBCN}}

\begin{definition}[\cite{Cheng2011IdentificationBCN}]
	BCN \eqref{BCN2} is called observable, if
there exists an input sequence $U\in(\Dt_M)^p$
for some $p\in\Z_{+}$, such that for any distinct states
$x_0,\bar x_0\in\Dt_N$,  $Hx_0=H\bar x_0$ implies
$(HL)_{x_0}^{p}(U)\ne (HL)_{\bar x_0}^{p}(U)$.
	\label{def7_observability}
\end{definition}

In this subsection, the observability of BCN \eqref{BCN2} refers to Definition \ref{def7_observability}.

According to Definition \ref{def7_observability},
to judge whether BCN \eqref{BCN2}
is observable, we need to
check the set $\V_n$ of all non-diagonal vertices
of its weighted pair graph
$(\V,\E,\W)$.

Now we design an algorithm to construct a DFA
for BCN \eqref{BCN2} according to its weighted pair graph $(\V,\E,\W)$.
The new DFA is denoted by $A_{\V_n}$, and accepts exactly every finite
input sequence by which not all non-diagonal state pairs can
be distinguished. The states of the DFA $A_{\V_n}$ are subsets of $\V$.

\begin{algo}\label{alg5_observability}
\begin{enumerate}
	\item\label{item1_obser2}
		Set $\Dt_M$ to be the alphabet of the DFA.
		Set the set $\V_n$ of all non-diagonal
		vertices of $\V$ to be
		the initial state of the DFA.
	\item\label{item2_obser2}
		For each letter $\dt_M^j$, $j\in[1,M]$, find the value
				for the transition partial
				function of the DFA at $(\V_n,\dt_M^j)$.
				The specific procedure is as follows:
				
				For each $j\in[1,M]$, let $s_j:=\{v\in\V|\text{there is }v'\in \V_n
				\text{ such that }(v',v)\in\E,\text{ and }\dt_M^j\in\W((v',v))\}$.
				If $s_j\ne\emptyset$, add $s_j$ to the state
				set of the DFA and set $s_j$ to be the value of the transition
				partial function at $(\V_n,\dt_M^j)$;
				otherwise, the transition partial function
				of the DFA is  not well defined at
				$(\V_n,\dt_M^j)$.
	\item\label{item3_obser2}
		For each new  state $s$ of the DFA found in the previous step,
				for each letter $\dt_M^j$, $j\in[1,M]$, find the value
				for the transition partial
				function of the DFA at $(s,\dt_M^j)$ according to
				Step \ref{item2_obser2}.
	\item\label{item4_obser2}
		Repeat Step \ref{item3_obser2} until
				no new state of the DFA occurs. (Since $\V$
				is a finite set, so is its power set,
				this repetition  will stop.)
	\item\label{item5_obser5}
		Set all the states of the obtained DFA  to be final states.
\end{enumerate}
\end{algo}

According to Algorithm \ref{alg5_observability}, the following
theorem holds.
\begin{theorem}
	BCN \eqref{BCN2} is not observable in the sense of Definition
	\ref{def7_observability} iff the DFA $A_{\V_n}$ generated by
	Algorithm \ref{alg5_observability} recognizes language $(\Dt_M)^*$.
	\label{thm5_observability}

\end{theorem}

\begin{proof}
	Notice that BCN \eqref{BCN2} is not observable iff none of finite input sequences
	can distinguish all state pairs of $\V_n$, that is, $L(A_{\V_n})=
	(\Dt_M)^*$.
\end{proof}

From Proposition \ref{prop6_observability},
Theorem \ref{thm5_observability} and Algorithm \ref{alg5_observability},
the following result which follows can be used to judge
whether BCN \eqref{BCN2} is observable.

\begin{theorem}\label{alg6_observability}
	BCN \eqref{BCN2} is not observable in the sense of Definition
	\ref{def7_observability} iff the DFA $A_{\V_n}$ generated by
	Algorithm \ref{alg5_observability} is complete.
\end{theorem}

\begin{example}\label{exam4_observability}
	Check whether BCN \eqref{eqn2_observability} is observable.
	
	According to Theorem \ref{alg6_observability}, we should check
	whether DFA
	$A_{\{(\dt_4^2,\dt_4^3),(\dt_4^2,\dt_4^4),
	(\dt_4^3,\dt_4^4)\}}$ is complete.

	From Fig. \ref{fig6:observability}, one sees that this DFA is complete.
	Then by Theorem \ref{alg6_observability}, BCN \eqref{eqn2_observability}
	is not observable.

\end{example}

\begin{figure}
		\centering
\begin{tikzpicture}[>=stealth',shorten >=1pt,auto,node distance=2.0 cm, scale = 1.0, transform shape,
	->,>=stealth,inner sep=2pt,state/.style={shape=circle,draw,top color=red!10,bottom color=blue!30},
	point/.style={circle,inner sep=0pt,minimum size=2pt,fill=},
	skip loop/.style={to path={-- ++(0,#1) -| (\tikztotarget)}}]
	\node[accepting,initial,state] (2*)  {$23,24,34$};
	\node[accepting,state] [above right of = 2*] (11) {$11$};
	\node[accepting,state] [below right of = 2*] (22) {$22$};
	\path [->] (2*) edge node {$2$} (11)
		   (2*) edge node {$1$} (22)
		   (22) edge [bend right] node {$2$} (11)
		   (22) edge [loop right] node {$1$} (22)
		   (11) edge [loop right] node {$1,2$} (11)
	;
		\end{tikzpicture}
	\caption{The DFA $A_{\V_{
	\{(\dt_4^2,\dt_4^3),(\dt_4^2,\dt_4^4),
	(\dt_4^3,\dt_4^4)\}	}}$ with respect to BCN
	\eqref{eqn2_observability} generated by Algorithm
	\ref{alg5_observability},
	where the number $ij$ in each circle denotes the state pair $(\dt_4^i,
	\dt_4^j)$, and the weight $k$ beside each edge denotes the input $\dt_2^k$.
	}
	\label{fig6:observability}
\end{figure}
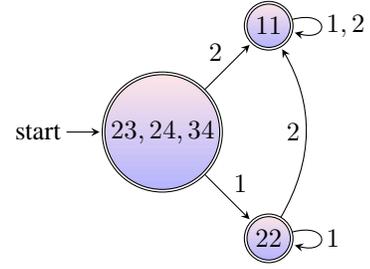

\begin{remark}
	In \cite{Laschov2013ObservabilityofBN:GraphApproach}, it is proved that
determining this observability  is {\bf NP}-hard. Actually, the results of
\cite{Laschov2013ObservabilityofBN:GraphApproach} show that determining each of
the four types of observability  is {\bf NP}-hard.
How to determine this observability  has been
solved in \cite{Li2013ObservabilityConditionsofBCN}
by enumerating all possible input sequences of a common finite length.
However, one can use our method to find any input sequence that determines
the initial state.
Due to the independence of initial states, their method cannot
be applied to deal with
Definitions \ref{def3_observability} or \ref{def4_observability}.
\end{remark}

\subsection{Determining the observability  in
\cite{Fornasini2013ObservabilityReconstructibilityofBCN}}

\begin{definition}[\cite{Fornasini2013ObservabilityReconstructibilityofBCN}]\label{def8_observability}
BCN \eqref{BCN2} is called observable, if for any distinct states
$x_0,\bar x_0\in\Dt_N$, for any input sequence $U\in(\Dt_M)^{\N}$,
$Hx_0=H\bar x_0$ implies
$(HL)_{x_0}^{\N}(U)\ne (HL)_{\bar x_0}^{\N}(U)$.
\end{definition}

In this subsection, the observability of BCN \eqref{BCN2}
means Definition \ref{def8_observability}.

According to Definition \ref{def8_observability}, BCN \eqref{BCN2} is not
observable iff there are two distinct states $\dt_N^i,\dt_N^j$ and an input
sequence $U\in(\Dt_M)^{\N}$ such that $H\dt_N^i=H\dt_N^j$ and $(HL)_{\dt_N^i}
^{\N}(U)=(HL)_{\dt_N^j}^{\N}(U)$.
Then the following theorem can be used to determine this observability.

\begin{theorem}
	BCN \eqref{BCN2} is not observable in the sense of Definition
	\ref{def8_observability} iff there is a non-diagonal vertex
	$(\dt_N^i,\dt_N^j)$ of the weighted pair graph of BCN \eqref{BCN2}
	such that the transition graph of the DFA
	$A_{(\dt_N^i,\dt_N^j)}$ generated by Algorithm \ref{alg3_observability}
	has a cycle.
	\label{thm6_observability}
\end{theorem}

\begin{proof}
	Since the transition graph has a finite number of vertices,
	the graph has a cycle iff there is an input sequence
	$U\in(\Dt_M)^{\N}$ such that $(HL)_{\dt_N^i}^{\N}(U)=
	(HL)_{\dt_N^j}^{\N}(U)$.
\end{proof}

In fact, one can determine the observability directly from the weighted pair
graph of BCN \eqref{BCN2}. Theorem \ref{thm6_observability} directly implies
the following result.

\begin{theorem}
	BCN \eqref{BCN2} is not observable in the sense of Definition
	\ref{def8_observability} iff there is a cycle in its weighted pair graph,
	and either the cycle contains a non-diagonal vertex, or there is a path
	from a non-diagonal vertex to the cycle.
	\label{thm11_observability}
\end{theorem}

\begin{example}\label{exam5_observability}
	Check whether BCN \eqref{eqn2_observability} is observable.

	By Theorem \ref{thm11_observability} and Fig. \ref{fig3:observability},
	BCN \eqref{eqn2_observability} is not observable.
\end{example}

\begin{remark}
An equivalent condition for this observability is given
in \cite{Fornasini2013ObservabilityReconstructibilityofBCN} by
checking each pair of distinct periodic state-input trajectories of the
same minimal period and  same length.
In addition, a specific critical length is given in \cite{Fornasini2013ObservabilityReconstructibilityofBCN}
such that if none of the input sequences of
that specific length can determine the initial
states, nor can input sequences of any other length.
Due to the independence of initial states and inputs, their method cannot
be used to deal with
Definitions \ref{def3_observability} or \ref{def4_observability} either.
\end{remark}

By the end of this subsection, we give a further result on this observability.

\begin{theorem}
	Consider BCN \eqref{BCN2}. Denote
	the number of non-diagonal vertices
	of its weighted pair graph by $N_{nd}$.
	The BCN is observable in the sense of Definition \ref{def8_observability}, iff
	$N_{nd}=0$ or,
	for all distinct states $x_0,\bar x_0\in\Dt_N$, for all input sequences
	$U\in(\Dt_M)^{N_{nd}}$, $Hx_0=H\bar x_0$ implies
	$(HL)_{x_0}^{N_{nd}}(U)\ne(HL)_{\bar x_0}^{N_{nd}}(U)$.
	\label{thm1:length:NewObservability}
\end{theorem}

\begin{proof}
	``if'':
	
	Obvious by Definition
	\ref{def8_observability}.

	``only if'':

	Assume that
	$N_{nd}>0$ and there are distinct $x,x'\in\Dt_N$ and an input sequence
	$U\in(\Dt_M)^{N_{nd}}$ such that
	$Hx=Hx'$ and $(HL)_{x}^{N_{nd}}(U)=(HL)_{x'}^{N_{nd}}(U)$.
	Use Algorithm \ref{alg3_observability} to generate DFA $A_{(x,x')}=(S,\Dt_M,\s,
	(x,x'),S)$. Then $U\in L(A_{(x,x')})$. Denote $\s( (x,x'),U)$ by $v_U$.
	If $v_U$ is diagonal, then $(HL)_{x}^{\N}(U(\dt_M^1)^{\infty})=
	(HL)_{x'}^{\N}(U(\dt_M^1)^{\infty})$, and the BCN is not observable.
	If $v_U$ is not diagonal, there are distinct $i,j\in[1,N_{nd}]$ such that either
	$\s( (x,x'),U[1,i])=\s( (x,x'),U[1,j])$ or $(x,x')=\s( (x,x'),U[1,j])$,
	for there are exactly $N_{nd}$ non-diagonal vertices.
	By Theorem \ref{thm11_observability}, the BCN is not observable.
\end{proof}

\section{Pairwise nonequivalence of the four types of observability of Boolean control networks}
\label{sec4}

In this section, we prove that no pairs of the four types of observability of BCNs
are equivalent, which reveals the essence of nonlinearity of BCNs (shown in Fig.
\ref{fig10:observability}).

\begin{theorem}
	If BCN \eqref{BCN2} is observable in the sense of Definition
	\ref{def3_observability}, then it is also observable in the sense of
	Definition \ref{def4_observability}. The converse is not true.
	\label{thm2_observability}
\end{theorem}

\begin{proof}
	The first part naturally follows from Definitions
	\ref{def3_observability} and \ref{def4_observability}.
	We use BCN \eqref{eqn2_observability} to prove the second part.

	First, we prove that BCN \eqref{eqn2_observability} is not observable
	in the sense of Definition \ref{def3_observability}.

	Denote $M:=\dt_4[1,1,2,1,2,4,1,1]W_{[2,4]}{\bf1}_2
	=\dt_2[1,2,2,1,1,1,4,1]{\bf1}_2=\left[
	\begin{smallmatrix}
		2 & 1 & 0 & 2\\ 0 & 1 & 1 & 0\\ 0 & 0 & 0 & 0\\
		0 & 0 & 1 & 0
	\end{smallmatrix}\right]$. Then for all $k\in\Z_{+}$, $M^k=\left[
	\begin{smallmatrix}
		* & * & * & *\\ * & * & * & *\\0 & 0 & 0 & 0\\0 & 0 & * & 0
	\end{smallmatrix}\right].$ By 
	\cite[Theorem 3.3]{Zhao2010InputStateIncidenceMatrix},
	BCN \eqref{eqn2_observability} is not controllable.
	So one cannot use the test criteria proposed in
	\cite{Cheng2009bn_ControlObserva}
	to check whether BCN \eqref{eqn2_observability} is observable.

	Next we prove that BCN \eqref{eqn2_observability} is not observable
	by showing that for state $\dt_4^2$, there is no input sequence
	such that the corresponding output sequence can determine it.
	We only need to consider states $\dt_4^3,\dt_4^4$,
	as $H\dt_4^1\ne H\dt_4^2$.
	Arbitrarily given an input sequence $U\in(\Dt)^{\N}$. If $U(0)=\dt_2^1$,
	then $L_{\dt_4^2}^1(\dt_2^1)=L_{\dt_4^3}^1(\dt_2^1)=\dt_4^2$.
	Then for each such $U$, $(HL)_{\dt_4^2}^{\N}(U)=
	(HL)_{\dt_4^3}^{\N}(U)$. Else if $U(0)=\dt_2^2$,
	then $L_{\dt_4^2}^1(\dt_2^2)=L_{\dt_4^4}^1(\dt_2^2)=\dt_4^1$.
	Then for each such $U$, $(HL)_{\dt_4^2}^{\N}(U)=
	(HL)_{\dt_4^4}^{\N}(U)$. Then
    BCN \eqref{eqn2_observability}
	is not observable in the sense of Definition \ref{def3_observability}.

	Second, we prove that BCN \eqref{eqn2_observability} is observable
	in the sense of Definition \ref{def4_observability}.
	We only need to check the state pairs $(\dt_4^2,\dt_4^3)$,
	$(\dt_4^2,\dt_4^4)$ and $(\dt_4^3,\dt_4^4)$.

	For $(\dt_4^2,\dt_4^3)$, $(HL)_{\dt_4^2}^1(\dt_2^2)=\dt_2^1\ne
	(HL)_{\dt_4^3}^1(\dt_2^2)=\dt_2^2$.

	For $(\dt_4^2,\dt_4^4)$, $(HL)_{\dt_4^2}^1(\dt_2^1)=\dt_2^2\ne
	(HL)_{\dt_4^4}^1(\dt_2^1)=\dt_2^1$.

	For $(\dt_4^3,\dt_4^4)$, $(HL)_{\dt_4^3}^1(\dt_2^1)=\dt_2^2\ne
	(HL)_{\dt_4^4}^1(\dt_2^1)=\dt_2^1$.

	Thus, BCN \eqref{eqn2_observability} is observable in the sense of
	Definition \ref{def4_observability}.
\end{proof}

\begin{theorem}
	If BCN \eqref{BCN2} is observable in the sense of Definition
	\ref{def8_observability}, then it is also observable in the sense of
	Definition \ref{def4_observability}. The converse is not true.
	\label{thm7_observability}
\end{theorem}

\begin{proof}
	The first part follows from Definitions
	\ref{def4_observability} and \ref{def8_observability}.
	We also use BCN \eqref{eqn2_observability} to prove the second part.

	We have proved that BCN \eqref{eqn2_observability} is observable
	in the sense of Definition \ref{def4_observability} in Theorem
	\ref{thm2_observability}.
	BCN \eqref{eqn2_observability} is not observable
	in the sense of Definition \ref{def8_observability}, because 
	$H\dt_4^2=H\dt_4^4=\dt_2[1,2,2,2]\dt_4^2=\dt_2^2$ and
	$(HL)_{\dt_4^2}^{\N}(\dt_2^2(\dt_2^1)^{\infty})=
	(HL)_{\dt_4^4}^{\N}(\dt_2^2(\dt_2^1)^{\infty})$.
\end{proof}

\begin{theorem}
	If BCN \eqref{BCN2} is observable in the sense of Definition
	\ref{def8_observability}, then it is also observable in the sense of 
	Definition \ref{def7_observability}. The converse is not true.
	\label{thm10_observability}
\end{theorem}

\begin{proof}
	Assume that a given BCN \eqref{BCN2} is observable in the sense of 
	Definition \ref{def8_observability}, then arbitrarily given
	$U\in(\Dt_M)^{\N}$,
	for any distinct $\dt_N^i,\dt_N^j$, $H\dt_N^i=H\dt_N^j$ implies
	$(HL)_{\dt_N^i}^{\N}(U)\ne(HL)_{\dt_N^j}^{\N}(U)$.
	Since $N<+\infty$, there is $p\in\Z_{+}$ such that for any distinct
	$\dt_N^i,\dt_N^j$, $H\dt_N^i=H\dt_N^j$ implies
	$(HL)_{\dt_N^i}^{p}(U[0,p-1])\ne(HL)_{\dt_N^j}^{p}(U[0,p-1])$.
	That is, the BCN is observable in the sense of
	Definition \ref{def7_observability}.

	To prove the second part, consider the following BCN:
	\begin{equation}
		\begin{split}
			x(t+1) &= \dt_4[1,1,3,3,1,2,3,2]x(t)u(t),\\
			y(t) &= \dt_2[1,1,2,2]x(t),
		\end{split}
		\label{eqn4_observability}
	\end{equation}
	where $t\in\N$, $x\in\Dt_4$, $y,u\in\Dt$.

	Choose $U=\dt_2^1\in(\Dt)^1$. $H\dt_4^1=H\dt_4^2=\dt_2^1$,
	$(HL)_{\dt_4^1}^1(U)=\dt_2^1\ne(HL)_{\dt_4^2}^1(U)=\dt_2^2$.
	$H\dt_4^3=H\dt_4^4=\dt_2^2$,
	$(HL)_{\dt_4^3}^1(U)=\dt_2^1\ne(HL)_{\dt_4^4}^1(U)=\dt_2^2$.
	Then BCN \eqref{eqn4_observability} is observable in the sense of
	Definition \ref{def7_observability}.

	Consider any $U\in(\Dt)^{\N}$ such that $U(0)=\dt_2^2$.
	Then $L_{\dt_4^3}^{\N}(U)=L_{\dt_4^4}^{\N}(U)$ and
	$(HL)_{\dt_4^3}^{\N}(U)=(HL)_{\dt_4^4}^{\N}(U)$. That is,
	BCN \eqref{eqn4_observability} is not observable in the sense of
	Definition \ref{def8_observability}.
\end{proof}

\begin{theorem}
	If BCN \eqref{BCN2} is observable in the sense of Definition
	\ref{def7_observability}, then it is also observable in the sense of
	Definition \ref{def3_observability}. The converse is not true.
	\label{thm8_observability}
\end{theorem}

\begin{proof}
	The first part holds naturally. To prove the second part,
	consider the following BCN:
	\begin{equation}
		\begin{split}
			x(t+1) &= \dt_4[1,1,1,3,1,2,3,2]x(t)u(t),\\
			y(t) &= \dt_2[1,1,2,2]x(t),
		\end{split}
		\label{eqn3_observability}
	\end{equation}
	where $t\in\N$, $x\in\Dt_4$, $y,u\in\Dt$.

	The weighted pair graph of BCN \eqref{eqn3_observability} is as shown
	in Fig. \ref{fig7:observability}.

\begin{figure}
        \centering
\begin{tikzpicture}[>=stealth',shorten >=1pt,auto,node distance=2.0 cm, scale = 1.0, transform shape,
	->,>=stealth,inner sep=2pt,state/.style={shape=circle,draw,top color=red!10,bottom color=blue!30},
	point/.style={circle,inner sep=0pt,minimum size=2pt,fill=},
	skip loop/.style={to path={-- ++(0,#1) -| (\tikztotarget)}}]
	\node[state] (11)                                 {$11$};
	\node[state] (22) [right of = 11]                    {$22$};
	\node[state] (12) [left of = 11]                 {$12$};
	\node[state] (34) [right of = 22]                 {$34$};
	\node[state] (33) [below of = 11]                 {$33$};
	\node[state] (44) [below of = 22]                 {$44$};

	\path [->] (11) edge [loop above] node {$1,2$} (11)
		   (22) edge node {$1$} (11)
		   (22) edge node {$2$} (33)
		   (12) edge node {$1$} (11)
		   (34) edge node {$2$} (22)
		   (44) edge node {$1$} (33)
		   (44) edge node {$2$} (22)
		   (33) edge node {$1$} (11)
		   (33) edge node {$2$} (22)
		   ;
        \end{tikzpicture}
	\caption{The weighted pair graph of BCN \eqref{eqn3_observability},
	where the number $ij$ in each circle denotes the state pair $(\dt_4^i,
	\dt_4^j)$, the weight $k_1,k_2,\dots$ beside each edge denotes the
	weight $\{\dt_2^{k_1},\dt_2^{k_2},\dots\}$ of the edge.}
	\label{fig7:observability}
\end{figure}
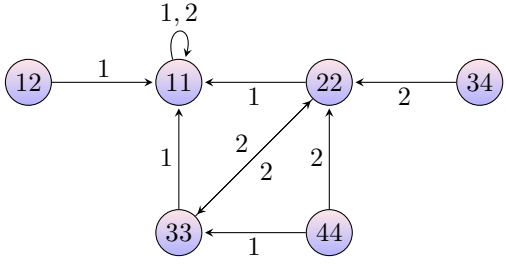

The DFA $A_{\{(\dt_4^1,\dt_4^2),(\dt_4^3,\dt_4^4)\}}$ generated by Algorithm
\ref{alg5_observability} (see Fig. \ref{fig8:observability}) is complete.
Then by Theorem \ref{thm5_observability}, BCN \eqref{eqn2_observability} is
not observable in the sense of Definition \ref{def7_observability}.

\begin{figure}
        \centering
\begin{tikzpicture}[>=stealth',shorten >=1pt,auto,node distance=2.0 cm, scale = 1.0, transform shape,
	->,>=stealth,inner sep=2pt,state/.style={shape=circle,draw,top color=red!10,bottom color=blue!30},
	point/.style={circle,inner sep=0pt,minimum size=2pt,fill=},
	skip loop/.style={to path={-- ++(0,#1) -| (\tikztotarget)}}]
	\node[accepting,initial,state] (2*)  {$12,34$};
	\node[accepting,state] [above right of = 2*] (11) {$11$};
	\node[accepting,state] [below right of = 2*] (22) {$22$};
	\node[accepting,state] [right of = 22] (33) {$33$};
	\path [->] (2*) edge node {$1$} (11)
		   (2*) edge node {$2$} (22)
		   (11) edge [loop above] node {$1,2$} (11)
		   (22) edge node {$1$} (11)
		   (22) edge node {$2$} (33)
		   (33) edge node {$2$} (22)
		   (33) edge node {$1$} (11)
	;
        \end{tikzpicture}
	\caption{The DFA $A_{\V_n 
	}$ with respect to BCN
	\eqref{eqn3_observability} generated by Algorithm
	\ref{alg5_observability},
	where the number $ij$ in each circle denotes the state pair $(\dt_4^i,
	\dt_4^j)$, the weight $k$ of each edge denotes the input $\dt_2^k$.
	}
	\label{fig8:observability}
\end{figure}

The DFAs $A_{\dt_4^1}$ and $A_{\dt_4^3}$ generated by Algorithm
\ref{alg1_observability} (see Fig. \ref{fig9:observability})
satisfy $\dt_2^2\notin L(A_{\dt_4^1})$ and
$\dt_2^1\notin L(A_{\dt_4^3})$. Then by Theorem \ref{thm3_observability},
BCN \eqref{eqn3_observability} is observable in the sense of Definition
\ref{def3_observability}.

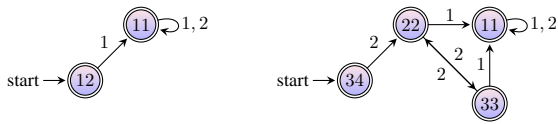
\begin{figure}
        \centering
\begin{tikzpicture}[>=stealth',shorten >=1pt,auto,node distance=1.5 cm, scale = 0.7, transform shape,
	->,>=stealth,inner sep=2pt,state/.style={shape=circle,draw,top color=red!10,bottom color=blue!30},
	point/.style={circle,inner sep=0pt,minimum size=2pt,fill=},
	skip loop/.style={to path={-- ++(0,#1) -| (\tikztotarget)}}]

	\tikzstyle{time}=[inner sep=0pt,minimum size=5mm]

	\node[accepting,initial,state] (2*)  {$12$};
	\node[accepting,state] [above right of = 2*] (11) {$11$};

	\path [->] (2*) edge node {$1$} (11)
		   (11) edge [loop right] node {$1,2$} (11)
	;

	\node[time] [above right of = 2*] (*) {};
	\node[time] [right of = *] (**) {};
	\node[time] [right of = **] (***) {};

	\node[accepting,initial,state] [below right of = ***] (3*) {$34$};
	\node[accepting, state] [above right of = 3*] (22*) {$22$};
	\node[accepting, state] [right of = 22*] (11*) {$11$};
	\node[accepting, state] [below of = 11*] (33*) {$33$};

	\path [->] (3*) edge node {$2$} (22*)
		   (22*) edge node {$2$} (33*)
		   (33*) edge node {$2$} (22*)
		   (22*) edge node {$1$} (11*)
		   (33*) edge node {$1$} (11*)
		   (11*) edge [loop right] node {$1,2$} (11*)
	;

        \end{tikzpicture}
	\caption{The DFAs $A_{\dt_4^1}$ and $A_{\dt_4^3}$
	with respect to BCN
	\eqref{eqn3_observability} generated by Algorithm
	\ref{alg1_observability},
	where the number $ij$ in each circle denotes the state pair $(\dt_4^i,
	\dt_4^j)$, the weight $k$ beside each edge denotes the input $\dt_2^k$.
	}
	\label{fig9:observability}
\end{figure}

\end{proof}

\begin{theorem}
	If BCN \eqref{BCN2} is observable in the sense of Definition
	\ref{def8_observability}, then it is also observable in the sense of
	Definition \ref{def3_observability}. The converse is not true.
	\label{thm9_observability}
\end{theorem}

\begin{proof}
	The first part holds naturally. To prove the second part,
	consider BCN \eqref{eqn3_observability} again.

	We have proved that BCN \eqref{eqn3_observability} is observable in the sense of
	Definition \ref{def3_observability} in Theorem
	\ref{thm8_observability}. Note that in Fig. \ref{fig9:observability},
	the DFAs are just the corresponding ones
	generated by Algorithm \ref{alg3_observability}. 
	Then By Theorem
	\ref{thm6_observability},
	BCN \eqref{eqn3_observability} is not
	observable in the sense of Definition \ref{def8_observability}.
\end{proof}

\section{Concluding remarks}\label{sec6}

In this paper, we solved the problem on determining the observability
of Boolean control networks (BCNs) completely
by using techniques in finite automata.
Also, we showed that no pairs of all the four types of observability notions are equivalent
by counterexamples, which
reveals the essence of nonlinearity of BCNs (shown in Fig. \ref{fig10:observability}).

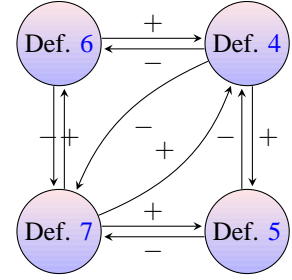
\begin{figure}
        \centering
\begin{tikzpicture} [>=stealth',shorten >=1pt,auto,node distance=2.5 cm, scale = 1.0, transform shape,
	->,>=stealth,inner sep=2pt,state/.style={shape=circle,draw,top color=red!10,bottom color=blue!30},
	point/.style={circle,inner sep=0pt,minimum size=2pt,fill=},
	skip loop/.style={to path={-- ++(0,#1) -| (\tikztotarget)}}]
	\node[state] (7) {Def. \ref{def7_observability}};
	\node[state] [right of =7] (4) {Def. \ref{def3_observability}};
	\node[state] [below of =4] (5) {Def. \ref{def4_observability}};
	\node[state] [left of =5] (8) {Def. \ref{def8_observability}};
	\draw [->] ([yshift=2pt] 7.east) -- node {$+$} ([yshift=2pt] 4.west);
	\draw [->] ([yshift=-2pt] 4.west) -- node {$-$} ([yshift=-2pt] 7.east);
	\draw [->] ([xshift=2pt] 4.south) -- node {$+$} ([xshift=2pt] 5.north);
	\draw [->] ([xshift=-2pt] 5.north) -- node {$-$} ([xshift=-2pt] 4.south);
	\draw [->] ([yshift=2pt] 8.east) -- node {$+$} ([yshift=2pt] 5.west);
	\draw [->] ([yshift=-2pt] 5.west) -- node {$-$} ([yshift=-2pt] 8.east);
	\draw [->] (8) to [out=22.5, in=249.5] node {$+$}  (4);
	\draw [->] (4) to [out=204.5, in=67.5] node  {$-$} (8);
	\draw [->] ([xshift=2pt] 8.north) -- node {$-$} ([xshift=2pt] 7.south);
	\draw [->] ([xshift=-2pt] 7.south) -- node {$+$} ([xshift=-2pt] 8.north);
        \end{tikzpicture}
	\caption{The implication relationships between Definitions
		\ref{def3_observability}, \ref{def4_observability},
	\ref{def7_observability} and \ref{def8_observability},
where ``$+$'' means ``implies'' and ``$-$'' means ``does not imply''.
	}
	\label{fig10:observability}
\end{figure}

Note that the computational complexity of algorithms for determining the first and
fourth types of observability is in exponential time, and the algorithms for the other two types
are in doubly exponential time.
How to reduce the computational complexity effectively is
a challenging and urgent problem, and we are naturally concerned with
``Is there a nondeterministic polynomial time algorithm for determining
the observability of BCNs?'' Furthermore, we conjecture that
``Determining the observablity of BCNs is {\bf PSPACE}-hard.''


\section*{Acknowledgment}
The first author is in debt to Prof. Jarkko Kari, Drs. Charalampos Zinoviadis,
Ville Salo and Ilkka T\"{o}rm\"{a} at the
University of Turku, Finland for fruitful discussions while visiting
the same university in 2013.
Both authors thank  Mr. Chuang Xu
at the University of Alberta, Canada, the
anonymous referees and the associate editor  for their
valuable comments that highly improve the presentation of this paper.

%
%

\ifCLASSOPTIONcaptionsoff
  \newpage
\fi



%

\end{document}